\documentclass[12pt]{article}
\usepackage[a4paper, margin=1.1in]{geometry}
\usepackage{amsmath,amsthm,mathtools}
\usepackage{amsfonts,amssymb,ragged2e}
\usepackage[raggedright]{titlesec}
\usepackage{graphicx}
\usepackage{float}
\usepackage{caption}
\usepackage{chngcntr}
\usepackage{fncylab}
\usepackage{verbatim}
\usepackage{cite}
 \theoremstyle{plain} 
\newtheorem{thm}{Theorem}[section]
\newtheorem{cor}[thm]{Corollary}
\newtheorem{lem}[thm]{Lemma}
\newtheorem{prop}{Proposition}[section]

\newtheorem{defn}{Definition}[section]

\theoremstyle{remark}
\newtheorem{rem}{Remark}[section]

\newtheorem{exam}{Example}[section]

\begin{document}

\title
{\bf{Orderenergetic, hypoenergetic and equienergetic graphs resulting from some graph operations}}
\author {\small Jahfar T K \footnote{jahfartk@gmail.com} and Chithra A V \footnote{chithra@nitc.ac.in} \\ \small Department of Mathematics, National Institute of Technology, Calicut, Kerala, India-673601}
\date{ }
\maketitle
\begin{abstract}
 A graph $G$ is said to be orderenergetic, if its energy equal to its order and it is said to be hypoenergetic if its energy less than its order. Two non-isomorphic graphs of same order are said to be equienergetic if their energies are equal.  In this paper, we  construct some new families of orderenergetic graphs,  
   hypoenergetic  graphs,  equienergetic graphs, equiorderenergetic graphs and equihypoenergetic graphs.
\end{abstract}

\hspace{-0.6cm}\textbf{AMS classification}: 05C50
\section{Introduction}
\par In this paper, we consider simple undirected  graphs. Let $G=(V,E)$ be a simple graph of order $p$ and size $q$ with vertex set $V(G)=\{v_{1}, v_{2},...,v_{p}\}$ and edge set $E(G)=\{e_{1}, e_{2},...,e_{q}\}$.  The adjacency matrix $A(G)=[a_{ij}]$ of the graph $G$ is a square symmetric matrix of order $p$ whose $(i,j)^{th}$ entry is defined by

 \[a_{i,j}=\begin{matrix}
 \begin{cases}
 1, & \text{if $v_i$ and $v_j$ are adjacent, }\\    
 0, & \text{otherwise .}
 \end{cases}
 \end{matrix}\] 
  The   eigenvalues $\lambda_{1},\lambda_{2},...,\lambda_{p}$ of the graph $G$ are defined as the eigenvalues of its adjacency matrix $A(G)$.  If $\lambda_{1},\lambda_{2},...,\lambda_{t}$ are the distinct eigenvalues of $G$, then the spectrum of $G$ can be written as $spec(G)=\begin{pmatrix}
\lambda_1&\lambda_2& ...&\lambda_t\\

m_1&m_2&...&m_t\\
\end{pmatrix}$, where $m_j$ indicates the algebraic multiplicity of the eigenvalue $\lambda_j$, $1\leq j\leq t$ of $G$.
The energy \cite{gutman1978energy} of the graph $G$ is defined as  $\varepsilon(G)=\displaystyle\sum_{i=1}^{p} |{\lambda_i}|$. More results on graph energy are reported in \cite{gutman1978energy,balakrishnan2004energy}. Two non-isomorphic graphs are said to be cospetral if they have the same spectrum, otherwise they are known as non-cospectral.
Two non-isomorphic graphs of the same order are said to be equienergetic if they have  the same energy\cite{ramane2004equienergetic}. A graph of order $p$ is said to be hyperenergetic if its energy is greater than $2(p-1)$, otherwise graph is non hyperenergetic. 
%
Graphs of order $p$  with energy equal to $p$ is called  orderenergetic graphs\cite{akbari2020orderenergetic}. The number of graphs whose energy equal to its order are relatively small. So we are trying to find new families of orderenergetic graphs. The orderenergetic graphs are studied in \cite{akbari2020orderenergetic}.\\The spectrum of complete bipartite graph $K_{p,p}$ is \[spec(K_{p,p})=\begin{pmatrix}
-p&0&p\\

1&2p-2&1\\
\end{pmatrix}.\]
Then $\varepsilon(K_{p,p})=2p$, so $K_{p,p}$ is orderenergetic for every $p$. So  our interest is to find the orderenergetic graphs other than $K_{p,p}$. In 2007, I.Gutman  et al. \cite{gutman2007hypoenergetic}  introduced the definition of hypoenergetic graphs. A graph is said to be hypoenergetic if its energy is less than its order, otherwise it is said to be non hypoenergetic. The properties of hypoenergetic graphs are discussed in detail \cite{gutman2011hyperenergetic,gutman2008hypoenergetic,gutman2007hypoenergetic}.
In the chemical literature there are many graphs for which the energy exceeds the order of graphs. In 1973, England and Ruedenberg published a paper \cite{england1973delocalization} in which they
asked 
“why does the graph energy exceed the
number of vertices?”. In 2007, Gutman \cite{gutman2007energy} had proved that
if the graph G is regular of any non-zero degree, then G is non hypoenergetic.  The orderenergetic and hypoenergetic graphs have several applications in theoretical chemistry. A graph is said to be integral if all of its eigenvalues are integers. 
The aim of this paper is to construct new families of  
orderenergetic, hypoenergetic and equienergetic graphs using some graph operations. \\   
\par The complement graph $\overline{G}$ of $G$ is a graph with vertex set same as that of $G$ and two vertices in $\overline{G}$ are adjacent only if they are not adjacent in $G$. We shall use the following notations throughout this paper, $C_p$, $K_p$, $P_m$ and  $K_{r,s}$  denotes cycle  on $p$ vertices, complete graph on $p$ vertices,  path on $m$ vertices and  complete bipartite graph on $r+s$ vertices respectively. The symbols $I_m$ and $J_m$  will stands for the identity matrix of order $m$ and $m\times m$ matrix with all entries are ones respectively.\\ 
\indent The rest of the paper is organized as follows. In Section 2, we state some previously known results that will be needed in the subsequent sections. In Section 3, we construct some orderenergetic graphs. 
In Section 4, some new families of hypoenergetic graphs are presented. In Section 5, an infinite family of equienergetic, equiorderenergetic and equihypoenergetic graphs are given.    
 	 
 \section{Preliminaries}
In this section, we recall the concepts of the $m$-splitting graph, the $m$-shadow graph and the $m$-duplicate graph of a graph and list some previously established results.
\begin{defn}\textnormal{\cite{cvetkovic1980Spectra1}}
The Kronecker product of two graphs  $G_1$ and $G_2$ is a graph $ G_1\times G_2$ with vertex set $V(G_1) \times V(G_2)$ and the vertices $(x_1,x_2)$ and $(y_1,y_2)$ are adjacent if and only if $(x_1,y_1)$ and $(x_2,y_2)$ are edges in  $G_1$ and $G_2$ respectively.
\end{defn}
\begin{defn}\textnormal{\cite{cvetkovic1980Spectra1}}
Let $A\in M_{ m \times n}(\mathbb{R}) $ and $B\in M_{ p \times q}(\mathbb{R}) $ be two matrices of order $m\times n$ and $p\times q$ respectively. Then the Kronecker product of $A$ and $B$ is defined as follows \[ A\otimes B=\begin{bmatrix}
a_{11}B & a_{12}B & a_{13}B & \dots  & a_{1n}B \\
a_{21}B & a_{22}B & a_{23}B & \dots  & a_{2n}B \\
\vdots & \vdots & \vdots & \ddots & \vdots \\
a_{m1}B & a_{m2}B & a_{m3}B & \dots  & a_{mn}B\\
\end{bmatrix}.\]
\end{defn}
\begin{prop}\textnormal{\cite{cvetkovic1980Spectra1}}
Let $A,B\in M_{n}(\mathbb{R})$ be two matrices of order $n$. Let $\lambda$ be an eigenvalue of matrix $A$ with corresponding eigenvector $x$ and	$\mu$ be an eigenvalue of matrix $B$ with corresponding eigenvector $y$, then $\lambda\mu$ is an eigenvalue of $A\otimes B$ with corresponding eigenvector $x\otimes y.$
\end{prop}
\begin{lem}\textnormal{\cite{balakrishnan2012textbook}}\label{lem2.1}.
	If $G_1$ and $G_2$ are any two graphs, then $\varepsilon(G_1\times G_2)=\varepsilon(G_1) \varepsilon(G_2)$.
\end{lem}	
\begin{defn}\textnormal{\cite{balakrishnan2012textbook}}.
	The join of graphs $G_1$ and $G_2$, $G_1\vee G_2$ is obtained from $G_1\cup G_2$ by joining every vertex of $G_1$ with every vertex of $G_2$.
\end{defn}

\begin{prop}\textnormal{\cite{cvetkovic1980Spectra1}}.\label{prop2.3}
	If $G_1$ is a $r_1$-regular graph with $n_1$ vertices and $G_2$ is a $r_2$-regular graph with $n_2$ vertices, then the characteristic polynomial of $G_1\vee G_2$ is given by $$\phi (G_1\vee G_2,x)=\frac{\phi(G_1,x)\phi (G_2,x)}{(x-r_1)(x-r_2)}[(x-r_1)(x-r_2)-n_1n_2].$$
\end{prop}
\begin{defn}\textnormal{\cite{vaidya2017energy}}
	Let $G$ be a $(p,q)$ graph.Then the $m$-splitting graph of a graph $G$, $spl_m(G)$ is obtained by adding to each vertex $v$ of $ G $  new $ m $ vertices say, $v_1,v_2,...,v_m$ such that $v_i$ ,$1\leq i \leq m$
	is adjacent to each vertex that is adjacent to $v$ in G. The adjacency matrix of $m$-splitting graph of $G$ is \[ A(spl_m(G)) =\begin{bmatrix}
	A(G) & A(G) &  A(G) & \dots  & A(G)\\
	A(G) & O & O & \dots  & O \\
	\vdots & \vdots & \vdots & \ddots & \vdots \\
	A(G) & O & O & \dots  & O\\
	\end{bmatrix}_{(m+1)p}.\]\\
   \end{defn}
	\begin{prop}\textnormal{\cite{vaidya2017energy}}
		Let $G$ be a $(p,q)$ graph. Then the energy of $m$-splitting graph of $G$ is, $\varepsilon (spl_m(G))=\sqrt{1+4m}\varepsilon(G).$
	\end{prop} 
\begin{defn}\textnormal{\cite{vaidya2017energy}}
	Let $G$ be a $(p,q)$ graph. Then the $m$-shadow graph $D_m(G)$ of a connected graph $G$ is obtained by taking m copies of $G$  say, $G_1,G_2,...,G_m$ then join each vertex $u$ in $G_i$ to the neighbors of the corresponding vertex $v$  in  $G_j,1\leq i\leq m,1\leq j\leq m.$
	The adjacency matrix of m-shadow graph of $G$ is  
	$A(D_m(G)) =J_m\otimes A(G) .$\\ Note that the number of vertices in $D_m(G)$ is $pm$. \\\\If $m=2$, then the graph $D_2(G)$ is called shadow graph of $G$. 
\end{defn}	
\begin{lem}\textnormal{\cite{vaidya2017energy}}\label{lem1}
	Let $G$ be any graph. Then $\varepsilon(D_m(G))=m\varepsilon(G).$ 
\end{lem}
\begin{defn} \textnormal{\cite{sampathkumar1973duplicate}}.
Let $G=(V,E)$ be a $(p,q)$ graph with vertex set $V$ and edge set $E$. Let $W$ be a set such that $V\bigcap W=\emptyset$, $|V|=|W|$ and $f:V\rightarrow W$ be bijective $($for  $a \in V$ we write $f(a)$ as $a'$ for convenience $)$. A duplicate graph of $G$ is $D(G)=(V_1,E_1)$, where the vertex set $V_1=V\cup W$ and the edge set $E_1$ of $D(G)$ is defined as, the edge ab is in $E$ if and only if both $ab'$ and $a'b$ are in $E_1$.
 \\In general m-duplicate graph $D^m(G)$ is defined as $D^m(G)=D^{m-1}(D(G))$. \\Note that the $m$-duplicate graph has $2^mp$ vertices and $2^mq$ edges.
\\\\Note 1.\textnormal{\cite{indulal2006pair}} Energy of the duplicate graph $D(G)$, $\varepsilon(D(G))=2\varepsilon(G).$ 
\end{defn} 
\section{Construction of orderenergetic graphs}
In this section, it is possible to construct an infinite family of orderenergetic graphs from the given orderenergetic graphs. Let $G$ and $H$ be orderenergetic graphs, then $G\cup H$ is orderenergetic. For example, the graph $K_{p,p}\cup mK_2$ is orderenergetic, but this graph is not connected. \\
\\The following theorems give some new methods to construct an infinite family of connected orderenergetic graphs.
\begin{thm}
	Let $G$ be a connected orderenergetic graph of order $p$. Then the m-shadow graph, $D_m(G)$ is a connected orderenergetic graph.
\end{thm}

\begin{proof}
	Since $G$ is orderenergetic, by Lemma \ref{lem1}, $\varepsilon(D_m(G))=m\varepsilon(G)=mp.$ So $D_m(G)$ is orderenergetic. Also the $m$-shadow graph of a connected graph is connected. 
\end{proof}
\begin{rem}
	Let $G$ be an orderenergetic graph. Then the $m$-shadow graph of a duplicate graph, $D_m(D(G))$ is  orderenergetic.
\end{rem}
%
%


\begin{thm}
	Let $G$ be an $r$-regular orderenergetic graph of order $p$. Then $G\vee \overline{K_n}$ is orderenergetic if and only if $n=4p-2r.$
\end{thm}
\begin{proof}
	Let $r = \lambda_{1}, \lambda_{2}, ...,\lambda_{p}$ be the eigenvalues of $G.$ Since $G$ is orderenergetic, we have, $$\displaystyle\sum_{i=2}^{p} |{\lambda_i}|=p-r.$$ From  Proposition \ref{prop2.3}, the characteristic polynomial of $G \vee \overline{K_n}$ is given by $$\phi (G\vee \overline{K_n},x)=x^{n-1}(x-\lambda_2)(x-\lambda_3)...(x-\lambda_p)(x^2-rx-np).$$ Let $\alpha$ and $\beta$ be the roots of the equation  $x^2-rx-np=0$. It is easy to observe
	that $\alpha$ and $\beta$ are of opposite sign. Without loss of generality we assume that $\alpha>0$ and $\beta<0$, then $\alpha+\beta=r$, $\alpha\beta=-np.$
	Thus the spectrum of $G \vee \overline{K_n}$ is \[spec(G \vee \overline{K_n})=\begin{pmatrix}
	0&\lambda_2&\lambda_3& ...&\lambda_p&\alpha&\beta\\
	n-1&1&1&...&1&1&1\\
	\end{pmatrix}.\]
	Hence 
	$$\varepsilon(G \vee \overline{K_n})=\displaystyle\sum_{i=2}^{p} |{\lambda_i}|+|\alpha|+|\beta|=\displaystyle\sum_{i=2}^{p} |{\lambda_i}|+\alpha-\beta.$$
	If $G \vee \overline{K_n}$ is orderenergetic, then 
	\begin{align*}\varepsilon(G \vee \overline{K_n})=p+n
	&\Longleftrightarrow p-r+\alpha-\beta=p+n\\&\Longleftrightarrow\alpha-\beta=n+r\\
	\end{align*} Also $\alpha+\beta=r$ and $\alpha-\beta=n+r$ implies that $\alpha=\frac{n+2r}{2}$ and $\beta=-\frac{n}{2}.$ 
	\begin{align*}
	\alpha\beta=-np&\Longleftrightarrow\alpha\beta=\frac{-n^2-2nr}{4}\\&\Longleftrightarrow-np=\frac{-n^2-2nr}{4}\\&\Longleftrightarrow 4np=n^2+2nr\\&\Longleftrightarrow n=4p-2r.
	\end{align*}.
\end{proof}
\begin{exam}
	Let $G=C_4$. Then $G\vee \overline{K_{12}}$ is orderenergetic.\\\[spec(G)=\begin{pmatrix}
	-2&0&2\\
	1&2&1\\
	\end{pmatrix}.\]\[spec(G\vee \overline{K_{12}})=\begin{pmatrix}
	-6&-2&0&8\\
	1&1&13&1\\
	\end{pmatrix}.\]\\$\varepsilon(G \vee \overline{K_{12}})=16$ and order of $G \vee \overline{K_{12}}$ is 16.
	\begin{exam}
		Consider $K_2 \vee \overline{K_{6}}, 
		$\begin{figure}[H]
			\centering
			\includegraphics[width=10.0cm]{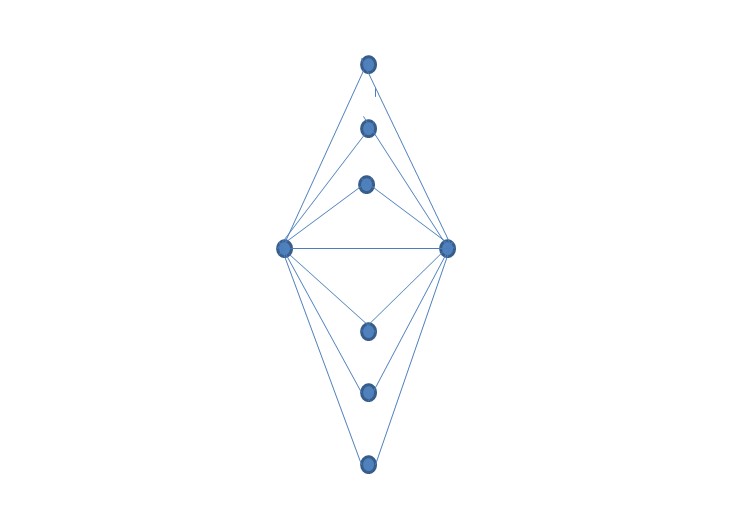}
			\caption{Graph 	$K_2 \vee \overline{K_{6}}$ }
			\label{pict26.jpg}
		\end{figure}\[spec(K_2 \vee \overline{K_{6}})=\begin{pmatrix}
		-3&-1&0&4\\
		1&1&5&1\\
		\end{pmatrix}.\]\\$\varepsilon(K_2 \vee \overline{K_{6}})=8$ and order of $K_2 \vee \overline{K_{6}}$ is 8.
	\end{exam}
	
\end{exam}
\begin{thm}
	Let $G$ be an orderenergetic graph with $p$ vertices. Then the $2$-splitting graph of $G$ , $spl_2(G)$ is orderenergetic.
\end{thm}

In\textnormal{\cite{stevanovic2020few}}, D.Stevanovic introduced the graph superpath $SP(a_1, a_2,..., a_m)$ obtained by replacing each  vertex $v_i$ of the path $P_m$ with totally disconnected graph $\overline K_{a_i}$. 
Two vertices $u\in \overline K_{a_i}$ and $w\in \overline K_{a_j}$ are adjacent in $SP(a_1, a_2,..., a_m)$  if $v_i$ and $v_j$ are adjacent in $P_m$, $i,j\in \{1,2,...,m\}$. The order of $SP(4, 1, 3, 2, 2, 3, 1, 4)$ is $m(m+1).$   \\\\ For example, the superpath $SP(4, 1, 3, 2, 2, 3, 1, 4).$ 
\begin{figure}[H]
	\centering
	\includegraphics[width=10.0cm]{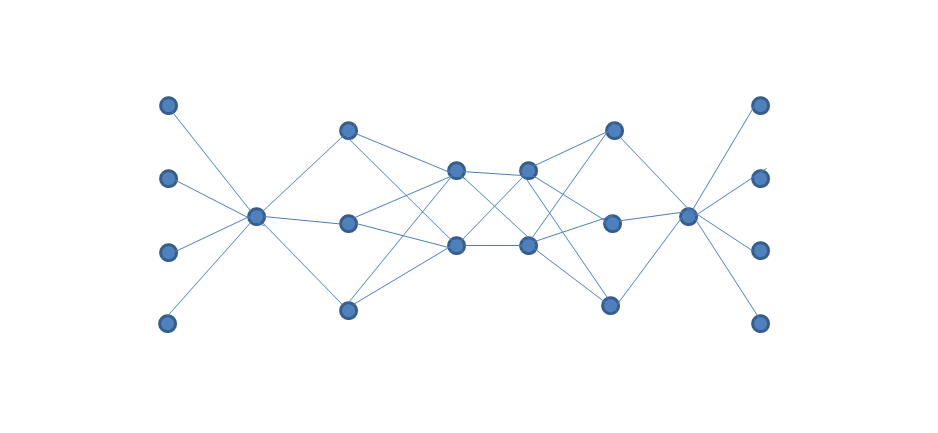}
	\caption{Graph $SP(4, 1, 3, 2, 2, 3, 1, 4).$ }
	\label{pict26.jpg}
\end{figure} 
Note that the maximum degree $\bigtriangleup$ of $SP(m, 1, m-1, 2,... 2, m-1, 1, m)$ is $2m-1$.
\begin{thm}\textnormal{\cite{stevanovic2020few}}
	The superpath $SP(m, 1, m-1, 2,..., 2, m-1, 1, m)$ is integral for each
	natural number m. Its spectrum consists of the simple eigenvalues $\pm m,\pm(m-1),\pm(m-2),
	...,\pm1$ and the eigenvalue 0 with multiplicity $m(m -1)$.
\end{thm} 
From this theorem, we can say that the eigenvalues of $SP(m, 1, m-1, 2,..., 2, m-1, 1, m)$  are consecutive integers  $\pm1,\pm2,\pm3,...,\pm m$.
\begin{cor}
	The energy of $SP(m, 1, m-1, 2,... 2, m-1, 1, m)$ is $m(m+1).$
\end{cor}
The following corollary gives the existence of orderenergetic graph of maximum  degree $2m-1 $.
\begin{cor}
	The graph $SP(m, 1, m-1, 2,... 2, m-1, 1, m)$ is orderenergetic for every $m$.\\
\textbf{Observation 1.} The graph $SP(m, 1, m-1, 2,..., 2, m-1, 1, m)$ is a graph with least maximum degree in the collection of all orderenergetic graphs having $m^2+m$ vertices. \\\\
\textbf{Observation 2.} Let $G$ be an orderenergetic graph. Then $G$ is an integral graph. 
\end{cor}

\section{Hypoenergetic graphs}
\par In 2007, I.Gutman et al.\cite{gutman2007hypoenergetic} introduced the definition of hypoenergetic graphs. In this section, we present some techniques for constructing sequence of hypoenergetic graphs.
\begin{prop} Kronecker product of two hypoenergetic graphs is hypoenergetic.
	\begin{proof}
		Let $G_1$ and $G_2$ be two hypoenergetic graphs with order $n_1$ and $n_2$ respectively. Then $\varepsilon(G_1)< n_1$ and $\varepsilon(G_2)< n_2$. By Lemma 2.1, $\varepsilon(G_1\times G_2)=\varepsilon(G_1) \varepsilon(G_2)<n_1n_2$. Thus Kronecker product of $G_1$ and $G_2$, $G_1 \times G_2$ is hypoenergetic graph.
	\end{proof} 

\end{prop}
The following theorem enable us to construct infinitely many hypoenergetic graphs. 
\begin{prop}
	Let $G_1$ be an orderenergetic graph  and $G_2$ be a hypoenergetic graph. Then Kronecker product of $G_1$ and $G_2$, $G_1 \times G_2$ is hypoenergetic graph .
\end{prop}
	\begin{proof}
	Let $G_1$ and $G_2$ be two  graphs with order $n_1$ and $n_2$ respectively. Since $G_1$ is orderenergetic,  $\varepsilon(G_1)= n_1$ and $G_2$ is hypoenergetic, $\varepsilon(G_2)< n_2$. By Lemma 2.1, $\varepsilon(G_1\times G_2)=\varepsilon(G_1) \varepsilon(G_2)<n_1n_2$. 
\end{proof} 
\begin{exam}
	Let $G=K_{p,p}\times K_{1,3}.$ Then $|V(K_{p,p}\times K_{1,3})|=8p$ and  $\varepsilon(K_{p,p}\times K_{1,3})=4\sqrt{3}p<8p$. So $K_{p,p}\times K_{1,3}$ is hypoenergetic for every $p$ . 
\end{exam}
\begin{prop}
	Let $G$ be a hypoenergetic graph of order  $p$. Then the $m$-shadow graph of $G$, $D_m(G)$ is hypoenergetic for every $m$ .
\end{prop}
\begin{exam}
	Let $G=K_{r,s}$  $r\ne s$. 
	 Then $|V(K_{r,s})|=r+s$ and  $\varepsilon(D_m(K_{r,s}))=2m\sqrt{rs}<m(r+s)=|V(D_m(K_{r,s}))|$ .
\end{exam}
\begin{rem}
	Let $G$ be a hypoenergetic graph. Then the $m$-shadow graph of duplicate graph,  $D_m(D(G))$ is  hypoenergetic.
	\end{rem}
\begin{prop}
	Let $G$ be a hypoenergetic graph of order  $p$. Then the $m$-splitting graph of $G$, $spl_m(G)$ is hypoenergetic for $m>2$ .
\end{prop}
\begin{proof}
	Since $G$ is a hypoenergetic graph, $\varepsilon(G)<p$. Also $|V(spl_m(G))|=p(m+1)$.\\ As $m>2$, \begin{align*}
	m(m-2)>0&\Longrightarrow(m+1)^2> 1+4m\\
	&\Longrightarrow \sqrt{1+4m}\varepsilon(G)<p(m+1)\notag\\
	&\Longrightarrow\varepsilon(spl_m(G))< p(m+1).
	\end{align*}
	Thus $spl_m(G))$ is hypoenergetic graph.
\end{proof}
Let $G$ be a graph of order $p$ and we denote  $G_r^s=K_{r,s}\times G, r,s\in N $. 

It is very interesting to construct hypoenergetic graphs from non hypoenergetic graphs. \\\par The following theorem describes a construction of hypoenergetic graphs from the complete graph. 
\begin{thm}
	Let $G=K_p$ be a complete graph on $p$ vertices and $m\geq 14$.	Then the graph $G_1^m$ is hypoenergetic.
	\begin{proof}The energy of complete graph is $2(p-1)$ and $|V(G_1^m)|=p(1+m)$.\\ As $m\geq 14$,
		\begin{align}
		m(m-14)+1>0
		&\Longrightarrow m^2+2m+1>16m\notag\\
		&\Longrightarrow4\sqrt{m}<(m+1)\notag	\\
		&\Longrightarrow4\sqrt{m}(p-1)<p(m+1)\notag\\
		&\Longrightarrow2\sqrt{m}\varepsilon(K_p)<p(m+1)\notag\\
		&\Longrightarrow\varepsilon(G_1^m)< p(1+m).\label{eq1}
		\end{align}
		Thus	$G_1^m$ is hypoenergetic whenever $m\geq14$.
	\end{proof}
\end{thm} 
\begin{rem}
	The graph $G=K_p$  is non hypoenergetic but $G_1^m$ is hypoenergetic for $m\geq14$. \\\par Since  $4\sqrt{m}>(m+1)$ for $m<14$, so inequality (\ref{eq1}) is satisfied for every  $p\leq k$, where $k=\lfloor\frac{4\sqrt{m}}{4\sqrt{m}-(m+1)}\rfloor$, $\lfloor x\rfloor$ is floor of x. Thus $G_1^m$, ($m<14$) is hypoenergetic whenever $p\leq k$.
\end{rem}
\begin{cor}
	Let $G$  be any non hyperenergetic graph. Then $G_1^m$ is hypoenergetic for every $m\geq14$. 
\end{cor}
	

\section{ Equienergetic graphs}

In this section, we construct some new pairs of equienergetic graphs.

\begin{prop}
	Let $G$ be a $(p,q)$ graph. Then the graphs $D_m(D(G))$ and $D_{2m}(G)$ are  non-cospectral equienergetic graphs. 
\end{prop}

\begin{prop}\label{3.4}
	Let $G$ be a $(p,q)$ graph and $D^m(G)$ be the m-duplicate graph of $G$. Then $\varepsilon (D^m(G))=2^m \varepsilon(G)$. 
\end{prop}

%
%

\begin{prop}The graphs $D^m(G)$ and $D_{2^m}(G)$ are  non-cospectral equienergetic graphs for all $m$.
\end{prop}
\begin{prop}
	Let $G$ be a simple $(p,q)$ graph. Then $G$ is integral if and only if its $m$-duplicate graph $D^m(G)$ is  integral.
\end{prop}
The following propositions describes the class of equiorderenergetic graphs.
\begin{prop}
	Let  $G$ be an orderenergetic graph of order $p$. Then the the graphs $spl_2(G)$ and $D_3(G)$ are equiorderenergetic graphs.
\end{prop}
\begin{figure}[H]
	\centering
	\includegraphics[width=10.0cm]{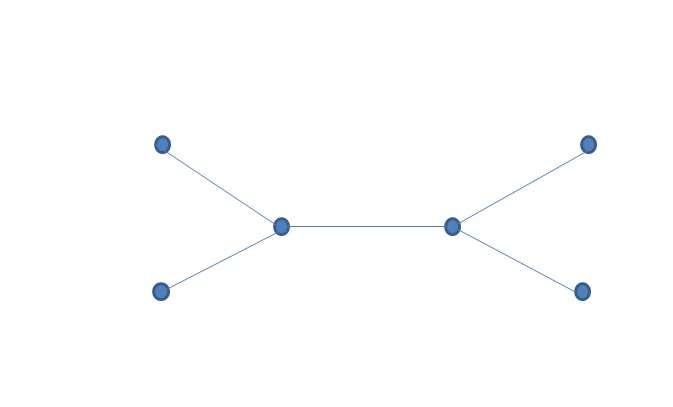}
	\caption{Graph $SP( 2, 1,1,2).$ }
	\label{pict26.jpg}
\end{figure}
\begin{figure}[H]
	\begin{minipage}[b]{0.5\linewidth}
		\centering
		\includegraphics[width=8.0 cm]{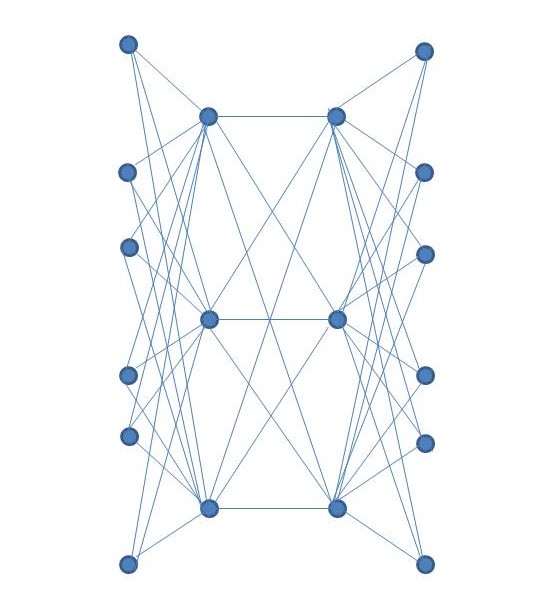}
		\caption{Graph $D_3(SP(2, 1, 1, 2)).$ }
		\label{pict12.jpg}
	\end{minipage}
	\hspace{0.5cm}
	\begin{minipage}[b]{0.5\linewidth}
		\centering
		\includegraphics[width=7.0 cm]{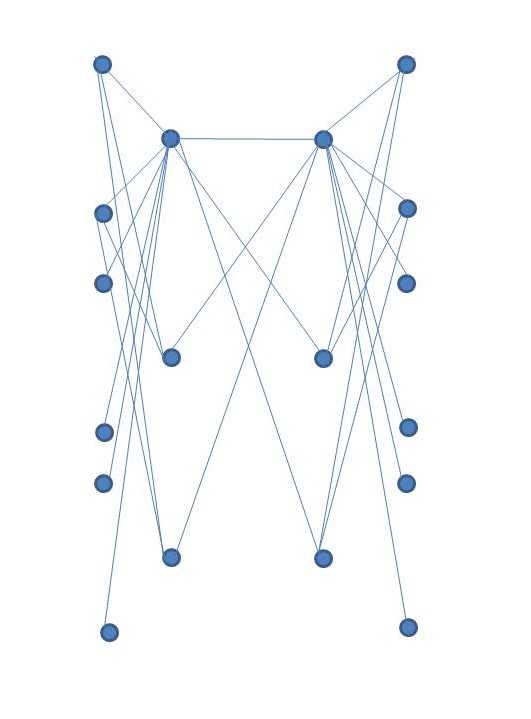}
		\caption{Graph  $spl_2(SP(2, 1, 1, 2)).$}
		\label{pict13.jpg}
	\end{minipage}
\end{figure}
The graphs $D_3(SP(2, 1, 1, 2))$ and $spl_2(SP(2, 1, 1, 2))$ are equiorderenergetic graphs.	
\begin{prop}
	Let  $G$ be a hypoenergetic graph. Then the graphs $spl_2(G)$ and $D_3(G)$ are equihypoenergetic graphs.
\end{prop}

\subparagraph{Conclusion}
In this paper, we  construct some family of orderenergetic graphs from the known orderenergetic graphs.  Also, some new families of hypoenergetic graphs are derived by using some graph operations. Moreover, the problem for constructing equienergetic graphs are discussed. In addition to that a new class of equiorderenergetic and equihypoenergetic graphs are obtained. 
\bibliography{refe1hypoorder}
\bibliographystyle{amsplain}

\end{document}